\def\draft{n}
\documentclass{amsart}
\usepackage[headings]{fullpage}
\usepackage{amssymb,epic,eepic,epsfig}
\usepackage{graphicx}
\usepackage{texdraw}

\theoremstyle{plain}

\newtheorem{theorem}{Theorem}

\newtheorem{proposition}{Proposition}[section]
\newtheorem{lemma}[proposition]{Lemma}
\newtheorem{corollary}[proposition]{Corollary}

\theoremstyle{definition}
\newtheorem{definition}[proposition]{Definition}

\theoremstyle{remark}
\newtheorem{example}[proposition]{Example}

\newtheorem{remark}[proposition]{Remark}

\def\printname#1{
        \if\draft y
                \smash{\makebox[0pt]{\hspace{-0.5in}
                        \raisebox{8pt}{\tt\tiny #1}}}
        \fi
}

\newcommand{\psdraw}[2]
         {\begin{array}{c} \hspace{-1.3mm}
        \raisebox{-4pt}{\epsfig{figure=draws/#1.eps,width=#2}}
        \hspace{-1.9mm}\end{array}}

\newlength{\standardunitlength}
\catcode`\@=11
\long\def\@makecaption#1#2{%
     \vskip 10pt

\setbox\@tempboxa\hbox{
       \small\sf{\bfcaptionfont #1. }\ignorespaces #2}%
     \ifdim \wd\@tempboxa >\captionwidth {%
         \rightskip=\@captionmargin\leftskip=\@captionmargin
         \unhbox\@tempboxa\par}%
       \else
         \hbox to\hsize{\hfil\box\@tempboxa\hfil}%
     \fi}
\font\bfcaptionfont=cmssbx10 scaled \magstephalf
\newdimen\@captionmargin\@captionmargin=2\parindent
\newdimen\captionwidth\captionwidth=\hsize
\catcode`\@=12

\def\lbl#1{\label{#1}\printname{#1}}

\def\BN{\mathbb N}
\def\BZ{\mathbb Z}

\def\BQ{\mathbb Q}

\def\BC{\mathbb C}

\def\calC{\mathcal C}

\def\D{\Delta}

\def\calH{\mathcal H}

\def\a{\alpha}

\def\La{\Lambda}
\def\l{\lambda}
\def\Ga{\Gamma}

\def\ga{\gamma}

\def\e{\epsilon}
\def\Ga{\Gamma}
\def\d{\delta}
\def\b{\beta}

\def\Om{\Omega}
\def\om{\omega}

\def\om{\omega}
\def\Om{\Omega}
\def\ep{\epsilon}
\def\ft{\mathfrak{t}}
\def\supp{\mathrm{supp}}

\begin{document}


\title[What is a sequence of Nilsson type?]{
What is a sequence of Nilsson type?}
\author{Stavros Garoufalidis}
\address{School of Mathematics \\
         Georgia Institute of Technology \\
         Atlanta, GA 30332-0160, USA \\ 
         {\tt http://www.math.gatech} \newline {\tt .edu/$\sim$stavros } }

\thanks{The author was supported in part by NSF. \\
\newline
1991 {\em Mathematics Classification.} Primary 57N10. Secondary 57M25.
}

\date{September 20, 2010}


\begin{abstract}
Sequences of Nilsson type appear in abundance in 
Algebraic Geometry, Enumerative Combinatorics, Mathematical Physics
and Quantum Topology. We give an 
elementary introduction on this subject, including the definition
of sequences of Nilsson type and the uniqueness, existence, and effective
computation of their asymptotic expansion.
\end{abstract}

\maketitle

\tableofcontents


\section{Sequences of Nilsson type: definition}
\lbl{sub.asexp}

Sequences of Nilsson type are the ones that are asymptotic to power series
in powers of $1/n$ and $\log n$. They appear in abundance 
Analysis (in asymptotic expansions of integrals), in Mathematical Physics
and in Algebraic Geometry (in relation to the Gauss-Manin connection);
see for example \cite{An2,M1,M2,Ph,Sa}. They also appear in Enumerative 
Combinatorics (see \cite{FS,WZ2,Ga2}) and in Quantum Topology.
For instance, 
the Witten-Reshetikhin-Turaev invariant of a closed 3-manifold is a sequence
of complex numbers that depends on the level, and it is expected to be 
of Nilsson type; see \cite{Wi,FG,Ga3,Ro,LR,AH}. 
In addition, the Kashaev invariant
of a knot is expected to be a sequence of Nilsson type; see \cite{KT,AH,CG}.
The quantum spin network evaluation at a fixed root of unity is known
to be a sequence of Nilsson type; see \cite{GV1,GV2}.
For a general discussion of perturbative and non-perturbative invariants
of knotted objects that are expected to be sequences of Nilsson type, 
see \cite{Ga1}.

There is a close connection between sequences
of Nilsson type and multivalued analytic functions with quasi-unipotent
monodromy; see for example Theorem \ref{thm.exists} below.

Several people familiar with the ideas of Quantum Topology have asked
for a self-contained definition of sequences of Nilsson type
and their asymptotics, its uniqueness, existence and effective computation.

Asymptotics of sequences is a well-studied subject of
analysis that goes back at least to Poincare; see for example \cite{O,C,M1}. 
Since we could not find a reference for sequences of Nilsson type
in the existing literature, we decided to write this introductory article.
It concerns the asymptotic
expansion of sequences which are relevant in Quantum Topology, and may
serve as an elementary introduction to asymptotics. We claim no original
results in this survey paper. 

In order to define sequences of Nilsson type, we need to introduce
Nilsson monomials $h_{\om}(n)$ indexed by a well-ordered set $\Om$, and
a finite set $\La$ of complex numbers of equal magnitude.

For a natural number $d \in \BN$, a finite subset $S$ of the rational 
numbers consider the well-ordered set $\Om=(S+\BN) \times \{0,1,\dots,d\}$ 
indexed by $(\a,\b) < (\a',\b')$ if and only if $\a <\a'$ or $\a=\a'$
and $\b'<\b$. $\Om$ has the order type of the natural numbers. In particular,
for every $\om \in \Om$, the set of elements strictly smaller than $\om$
is finite. 
Consider the $\Om$-indexed family of monomials of {\em Nilsson type}
given by:
\begin{equation}
\lbl{eq.hab}
h_{\om}(n)=\frac{(\log n)^{\b}}{n^{\a}}
\end{equation}
for $\om=(\a,\b) \in \Om$. It is easy to see that 
$\lim_n h_{\om'}(n)/h_{\om}(n)=0$ (abbreviated by $h_{\om'}(n)/h_{\om}(n)=o(1)$,
and also by $h_{\om}(n) \gg h_{\om'}(n)$) 
if and only if $\om < \om'$. This and all limits below are taken when
$n$ goes to infinity.

Fix a finite set $\La$ of nonzero complex
numbers of magnitude $r>0$. Let $c_{\om,\l}$ be a collection of complex numbers 
indexed by $\Om \times \La$.

\begin{definition}
\lbl{def.asexp} 
\rm{(a)}
With the above notation, for a complex-valued sequence $(a_n)$ the expression
\begin{equation}
\lbl{eq.asexp}
a_n \sim \sum_{\om \in \Om} h_{\om}(n) \sum_{\l \in \La} c_{\om,\l} \l^n
\end{equation}
means that 
\begin{itemize}
\item
for every $\om \in \Om$ we have:
\begin{equation}
\lbl{eq.asexp1}
\left( a_n r^{-n} -\sum_{\om' \leq \om} h_{\om'}(n) 
\sum_{\l \in \La} c_{\om',\l} (\l r^{-1})^n \right)
\frac{1}{h_{\om}(n)}=o(1) 
\end{equation}
\item
$c_{\om,\l} \neq 0$ for some $(\om,\l) \in \Om \times \La$.
\end{itemize}
\rm{(b)} We say that a sequence $(a_n)$ is of Nilsson type if there
exist $\Om, \La$ and $c_{\om,\l}$ such that \eqref{eq.asexp} holds.
\end{definition}

We will say that an asymptotic expansion \eqref{eq.asexp}
is $\Om\times\La$-{\em minimal} if 
\begin{itemize}
\item
For every $\l \in \La$ there
exists $\om \in \Om$ such that $c_{\om,\l} \neq 0$. 
\item
For every $\om \in \Om$ there exists $\l \in \La$ such that $c_{\om,\l}
\neq 0$.
\end{itemize}
By considering a subset of $\La$ or $\Om$ if necessary, it is easy to
see that every asymptotic expansion has a minimal representative.

\section{Uniqueness}
\lbl{sub.unique}

Our first task is to show that a 
sequence of Nilsson type uniquely determines $\Om$, $\La$ and the 
coefficients $c_{\om,\l}$. The key idea is the following elementary lemma.

\begin{lemma}
\lbl{lem.1}
If $\La$ is a finite subset of the unit circle and 
\begin{equation}
\lbl{eq.lambdas}
\sum_{\l \in \La} c_{\l} \l^n=o(1)
\end{equation}
holds for some complex numbers $c_{\l}$, 
then $c_{\l}=0$ for all $\l \in \La$.
\end{lemma}

\begin{proof}
Divide \eqref{eq.lambdas} by $\l_1^n$ for some $\l_1 \in \Lambda$. 
Then we have
$ c_{\l_1} + \sum_{\l \neq \l_1} c_{\l} (\l/\l_1)^n=o(1)$, where $\l/\l_1 \neq 1$.
So,
$$
\frac{1}{n} \sum_{k=1}^n
\left(c_{\l_1} + \sum_{\l \neq \l_1} c_{\l} (\l/\l_1)^k\right)=o(1).
$$
By averaging, it follows that
$$
c_{\l_1} +  \frac{1}{n} \sum_{\l \neq \l_1} c_{\l} 
\frac{1- (\l/\l_1)^{n+1}}{1-\l/\l_1}=o(1).
$$
Thus, $c_{\l_1}=0$. Since $\l_1$ was an arbitrary element of $\La$, the
result follows.
\end{proof}

\begin{lemma}
\lbl{lem.2}
If $(a_n)$ satisfies \eqref{eq.asexp} then
\begin{equation}
\lbl{eq.r}
\limsup_n |a_n|^{1/n}=r.
\end{equation}
\end{lemma}

\begin{proof}
Since $c_{\om,\l} \neq 0$ for some $(\om,\l) \in \Om \times \La$, 
without loss of generality assume that $c_{\om_0,\l} \neq 0$ for some 
$\l \in \La$ where $\om_0$ is the smallest element of $\Om$.
Equation \eqref{eq.asexp} for $\om=\om_0$ gives that
\begin{equation}
\lbl{eq.asr}
\frac{a_n r^{-n}}{h_{\om_0}(n)}  -\sum_{\l \in \La} c_{\om_0,\l} (\l r^{-1})^n
=o(1)
\end{equation}
Now $\l r^{-1}$ are on the unit circle. It follows that
$$
\left|\frac{a_n r^{-n}}{h_{\om_0}(n)}\right| < C
$$
for some $C>0$. Since $\lim_n h_{\om}(n)^{1/n}=1$ for all $\om \in \Om$,
it follows that
$$
\limsup_n |a_n|^{1/n} \leq r
$$
Since some $c_{\om_0,\l}$ is nonzero and $\l r^{-1}$ are on the unit circle,
Lemma \ref{lem.1} implies that $\lim_n \sum_{\l \in \La} c_{\om_0,\l} (\l r^{-1})^n
\neq 0$. Since the sequence is bounded, it follows that 
there exists a subsequence $n_k$ such that 
$$
\lim_{n_k} 
\sum_{\l \in \La} c_{\om_0,\l} (\l r^{-1})^{n_k} =C' \neq 0
$$
Combined with Equation \eqref{eq.asr}, it follows that
$$
\lim_{n_k} |a_{n_k}|^{1/{n_k}}= r
$$
The result follows.
\end{proof}
In particular, Lemma \ref{lem.2} implies 
that sequences of Nilsson type satisfy $\limsup_n |a_n|^{1/n}>0$.

\begin{proposition}
\lbl{prop.unique}
Suppose that 
\begin{equation}
\lbl{eq.asexpa}
a_n \sim \sum_{\om \in \Om} h_{\om}(n) \sum_{\l \in \La} c_{\om,\l} \l^n
\end{equation}
and
\begin{equation}
\lbl{eq.asexpb}
a_n \sim \sum_{\om' \in \Om'} h_{\om'}(n) \sum_{\l' \in \La'} c'_{\om',\l'} \l'^n
\end{equation}
are $\Om\times\La$-minimal and $\Om'\times\La'$-minimal asymptotic expansions.
Then $\Om=\Om'$, $\La=\La'$. Moreover, for all $(\om,\l) \in \Om \times \La$
we have $c_{\om,\l}=c'_{\om,\l}$.
\end{proposition}

\begin{proof}
Let $\om_0$ and $\om'_0$ denote the smallest elements of $\Om$ and 
$\Om'$. Lemma \ref{lem.2} implies that $r=r'$ where $r$ and $r'$ are
the magnitudes of the elements of $\La$ and $\La'$ respectively.
Equation \eqref{eq.asexp1} for $\om_0$ and $\om'_0$ implies that
\begin{equation}
\lbl{eq.ar0}
\frac{a_n r^{-n}}{h_{\om_0}(n)}  -\sum_{\l \in \La} c_{\om_0,\l} (\l r^{-1})^n
=o(1),
\qquad
\frac{a_n r^{-n}}{h_{\om'_0}(n)}  -\sum_{\l' \in \La'} c'_{\om'_0,\l'} (\l' r^{-1})^n
=o(1)
\end{equation}
If $\om_0 \neq \om'_0$, we may assume that $\om_0 < \om'_0$. In that
case, observe that $h_{\om'_0}(n)/h_{\om_0}(n)=o(1)$. 
Multiply the second equation above by $h_{\om'_0}(n)/h_{\om_0}(n)$
and subtract from the first. It follows that
$$
-\sum_{\l \in \La} c_{\om_0,\l} (\l r^{-1})^n +
\frac{h_{\om'_0}(n)}{h_{\om_0}(n)} 
\sum_{\l' \in \La'} c'_{\om'_0,\l'} (\l' r^{-1})^n =o(1)
$$
Since $h_{\om'_0}(n)/h_{\om_0}(n)=o(1)$, it follows that
$$
\sum_{\l \in \La} c_{\om_0,\l} (\l r^{-1})^n=o(1)
$$
Lemma \ref{lem.1} implies that $c_{\om_0,\l}=0$ for all $\l$ contrary
to our minimality assumption of \eqref{eq.asexpa}. 
It follows that $\om_0=\om'_0$. Subtracting, 
Equation \eqref{eq.ar0} implies that
$$
-\sum_{\l \in \La} c_{\om_0,\l} (\l r^{-1})^n +
\sum_{\l' \in \La'} c'_{\om'_0,\l'} (\l' r^{-1})^n =o(1)
$$
Lemma \ref{lem.1} implies that if $c_{\om_0,\l} \neq 0$ for some $\l \in \La$,
then $\l \in \La'$ and moreover $c_{\om_0,\l}=c'_{\om_0,\l'}$.

An easy induction on $\om \in \Om$ proves the following statement.
For every $\om \in \Om$, the following holds. If $c_{\om,\l} \neq 0$ 
for some $\l \in \La$, then $\l \in \La'$ and $\om \in \Om'$
and $c_{\om,\l}=c'_{\om,\l}$.

The minimality assumption and the above statement implies that $\Om=\Om'$ 
and $\La=\La'$ and $c_{\l,\om}=c'_{\l,\om}$ for all $(\om,\l) \in \Om \times \La$.
\end{proof}

\begin{remark}
\lbl{rem.1}
Proposition \ref{prop.unique} proves uniqueness in a non-effective way. 
We will come back to the problem of computing $c_{\om,\l}$ later on.
\end{remark}

\section{Alternative expression for sequences of Nilsson type}
\lbl{sub.remarks}

If $(a_n)$ is a sequence of Nilsson type, we can write
\eqref{eq.asexp} in the form:
\begin{equation}
\lbl{eq.nilsson}
a_n \sim \sum_{\l,\a,\b} \l^{n} n^{\a} (\log n)^{\b} S_{\l,\a,\b}
g_{\l,\a,\b}(1/n)
\end{equation}
where 
\begin{itemize}
\item
the summation in \eqref{eq.nilsson} is over a finite set, 
\item
the {\em growth rates} $\l$ are complex numbers numbers of equal magnitude,
\item
the {\em exponents} $\a$ are rational numbers and the 
{\em nilpotency exponents} $\b$ are natural numbers,
\item
the {\em Stokes constants} $S_{\l,\a,\b}$ are complex numbers,
\item
$g_{\l,\a,\b}(x) \in \BC[[x]]$ are formal power series in $x$ with complex
coefficients and leading term $1$.
\end{itemize}

\begin{remark}
\lbl{rem.nilsson}
In the definition of a sequence of Nilsson type, we may 
additionally require that
\begin{itemize}
\item
$\La$ is a set of algebraic numbers,
\item
the formal power series $g_{\l,\a,\b}(x)$ is {\em Gevrey-1}, i.e., that
the coefficient of $x^k$ in $g_{\l,\a,\b}(x)$ is bounded by $C^k k!$
for all $k$, where $C$ depends on $g_{\l,\a,\b}$,
\item
the coefficients of the
formal power series $g_{\l,\a,\b}(x)$ lie in a fixed number field $K$, 
\end{itemize}
These additional requirements hold for the evaluations of classical
spin networks, see \cite{GV1}, as well as Sections 
\ref{sub.existence} and \ref{sub.compute} below.
\end{remark}

\begin{example}
\lbl{ex.1}
For example, if $d=1$ and $S=\{1,3/2\}$, then $\Om=(1+\BN) \cup(3/2+\BN)$
and we have:
$$
\frac{\log n}{n} \gg \frac{1}{n} \gg  
\frac{\log n}{n^{3/2}} \gg \frac{1}{n^{3/2}} \gg
\frac{\log n}{n^2} \gg \frac{1}{n^2} \gg \dots
$$
If $\La=\{\kappa,\mu,\nu\}$, the asymptotic expansion \eqref{eq.nilsson}
of a sequence of Nilsson type becomes:
\begin{eqnarray*}
a_n & \sim & \frac{\log n}{n} 
\sum_{\l \in \La} \l^n S_{\l,1} g_{\l,1}\left(\frac{1}{n}\right)
+
\frac{\log n}{n^{3/2}} 
\sum_{\l\in\La} \l_j^n S_{\l,2} g_{\l,2}\left(\frac{1}{n}\right)
\\
& & +\frac{1}{n} 
\sum_{\l\in\La} \l_j^n S_{\l,3} g_{\l,3}\left(\frac{1}{n}\right)
+
\frac{1}{n^{3/2}} 
\sum_{\l\in\La} \l_j^n S_{\l,4} g_{\l,4}\left(\frac{1}{n}\right)
\end{eqnarray*}
where $g_{\l,j}(x) \in \BC[[x]]$ are formal power series in $x$
and $S_{\l,j}$ are complex numbers.
\end{example}

\section{Existence}
\lbl{sub.existence}

In this section we will prove that a sequence is of Nilsson type, 
under some analytic continuation assumptions of its generating series.
This is a well-known argument (see for example, \cite{C,CG,FS,GM,GIKM,M1}) 
that consists of the following parts:
\begin{itemize}
\item
apply Cauchy's theorem to give an integral representation of
the sequence,
\item
deform the contour of integration to localize the computation near
the singularities of the generating function,
\item
analyse the local computation using the local monodromy assumption
of the generating function.
\end{itemize}
Let us give the details of the above existence proof. 
Since sequences of Nilsson type are exponentially bounded (as follows
from Lemma \ref{lem.2}), fix an
exponentially bounded sequence $(a_n)$ and consider its generating series
\begin{equation}
\lbl{eq.Gz}
G(z)=\sum_{n=0}^\infty a_n z^n
\end{equation}
$G(z)$ is an analytic function for all complex numbers $z$ that satisfy
$|z|<1/R$. 
Suppose now that $G$ has analytic continuation on a disk of radius 
$r$ with singularities at finitely many points $\kappa,\l,\mu,\nu,\dots$.
Suppose also that $G$ has further analytic continuation on a disk of
radius $r+\e$ minus finitely many segments emanating from the
singularities radially as in the following figure.
$$
\psdraw{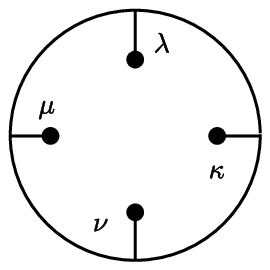}{1.5in}
$$
Assume in addition that $G$ has {\em quasi-unipotent local monodromy} 
at each singularity $\l,\mu,\nu,\kappa$ on the circle of radius $r$
(i.e., the eigenvalues of the local monodromy are complex roots of unity).

\begin{theorem}
\lbl{thm.exists}
Under the above assumptions, the sequence $(a_n)$ is of Nilsson type.
\end{theorem}

\begin{corollary}
\lbl{cor.exists}
Suppose that $G(z)=\sum_{n=0}^\infty a_n z^n$ is a multivalued analytic function
on $\BC\setminus\La$ (where $\La \subset \BC$ is a finite set) which
is regular at $z=0$, and has quasi-unipotent local monodromy.
 Then, $(a_n)$ is a sequence of Nilsson type.
\end{corollary}

\begin{remark}
\lbl{rem.exists}
We know of at least three ways to show that a germ $G(z)$ of an analytic
function can be analytically continued to the complex plane,
namely 
\begin{itemize}
\item[(a)]
$G$ satisfies a linear differential equation, see for example 
\cite[Thm.1]{Ga2} reviewed in Section \ref{sub.examples} below.
For examples that come from Quantum Topology (specifically, spin networks)
see \cite{GV1,GV2}.
\item[(b)]
$G$ satisfies a nonlinear differential equation. See for example
the instanton solutions of Painlev\'e I studied in detail in 
\cite{GIKM} and the matrix models of \cite{GM}.
\item[(c)]
$G$ is resurgent. See for example the Kontsevich-Zagier series studied
in detail in \cite{CG}, and more generally the arithmetic resurgence
conjecture of \cite{Ga1} for sequences that appear in Quantum Topology.
\end{itemize}
\end{remark}

\begin{proof}(of Theorem \ref{thm.exists})
We begin by  applying Cauchy's theorem to give an integral
representation of $(a_n)$. If $\ga$ is a circle of radius 
less than $1/R$ that contains the origin, then we have:

\begin{equation}
\lbl{eq.intan}
a_n=\frac{1}{2 \pi i} \int_{\ga} \frac{G(z)}{z^{n+1}}dz
\end{equation}
We can deform $\ga$ to a contour $\calC$ which consists of 
a contour $\calH_{\l}$ around each singularity $\l$ and finitely many arcs
$\ga_{r+\e}$ of the circle of radius $1/(r+\e)$ as in the following figure.
\begin{equation}
\lbl{eq.hankel}
\psdraw{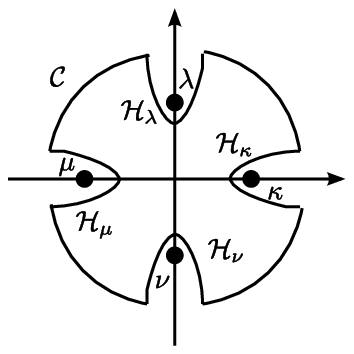}{2in}
\end{equation}
The contours $\calH_{\l}$ are known as {\em Hankel contours} in Analysis 
(see \cite{O})
and {\em Lefschetz thimbles} in Algebraic Geometry (see \cite{Ph,Sa}). 
Cauchy's theorem implies that
\begin{equation}
\lbl{eq.klm}
a_n= 
\frac{1}{2 \pi i} \sum_{\l \in \La}
\int_{\calH_{\l}} \frac{G(z)}{z^{n+1}}dz + 
\frac{1}{2 \pi i} 
\int_{\ga_{r+\e}} \frac{G(z)}{z^{n+1}}dz
\end{equation}
The above expression is exact, and decomposes the sequence
$(a_n)$ into a finite sum of sequences (one per singularity) and an extra 
term. Of course, there is nothing canonical about this decomposition, since 
the size of the Hankel contour and $\ga_{r+\e}$ depends on $\e$. One could
make the decomposition nearly canonical by using Hankel contours that 
extend to infinity, but even so there are choices of directions to be made, 
and we will not use them here.

The integral over $\ga_{r,\e}$ can be estimated by $O((r+\e)^{-n})$ since
$G$ is uniformly bounded on the arcs $\ga_{r,\e}$.
Since we assume that the local monodromy of
$G(z)$ around a singularity is {\em quasi-unipotent}, it follows 
(see \cite{M1}) that modulo germs of holomorphic functions at zero,
$G(\l+z)$ has a local expansion of the form
$$
G(\l+z)=
\sum_{\a',\b'} z^{\a'} (\log z)^{\b'} h_{\a',\b'}(z)
$$
where the summation is over a finite set, $\a'\in \BQ$,
$\b' \in \BN$ and $ h_{\a',\b'}(z)$ are germs 
of functions analytic at $z=0$. For a germ $f(z)$ of a {\em multi-valued} 
analytic function at $z=0$, let $\D_0 f$ denote its {\em variation}
defined by
$$
\D_0 f (z)= \lim_{\e\to 1} f(e^{2 \pi i \e}z)-
\lim_{\e\to 0} f(e^{2 \pi i \e}z)
$$
(see \cite{M1}) when $z$ is restricted on a line segment $[0,\ep)$.
The variation of the building blocks $z^{\a}$ and $(\log z)^{\b}$
are given by
$$
\D_0 (z^{\a})= 
\begin{cases} (e^{2 \pi i \a}-1) z^{\a} & \a \in \BQ\setminus\BZ \\
\d_0 & \a \in \BZ\setminus\BN \\
0 & \a \in \BN
\end{cases},
\qquad
\D_0 (\log z)=2 \pi i
$$
where $\d_0$ is the Dirac delta function (really, a distribution).
For a singularity $\l$ of $G(z)$, let $\D_{\l}G(z)$ denote the variation
of $G(\l+z)$. It follows that for $z$ in the line segment of Figure 
\eqref{eq.hankel}, we have
\begin{equation}
\lbl{eq.im3}
\D_{\l}G(z)= \sum_{\a,\b} z^{\a} (\log z)^{\b}
\sum_{k=0}^\infty c'_{\a,\b,\l,k} z^{k-1}
\end{equation}
where the sum is over a finite set $\{\a,\b\}$, $\a\in\BQ\setminus\BN$,
$\b\in\BN$ and $\sum_{k=0}^\infty \sum_{k=0}^\infty c'_{\a,\b,\l,k} z^k$ are germs
of analytic functions at $z=0$. When $\a \in \BZ\setminus\BN$, we can deform
the Hankel contour into a small circle centered around $\l$ and 
apply Cauchy's theorem. For the remaining cases $\a \in \BQ\setminus\BZ$,
a change of variables $z \mapsto \l(1+z)$ centers the Hankel contour at zero
and implies that
\begin{equation}
\lbl{eq.im2}
\int_{\calH_{\l}} \frac{G(z)}{z^{n+1}}dz =
\l^{-n} \int_{\calH_0} \frac{G(\l(1+z))}{(1+z)^{n+1}} d z=
\l^{-n} \int_0^{\e} \frac{\D_{\l}G(\l z)}{(1+z)^{n+1}} d z
\end{equation}
A {\em Beta-integral} calculation gives that
$$
\int_0^\infty \frac{z^{\ga-1}}{(1+z)^{n+1}} d z
=\frac{\Ga(\ga)\Ga(n+1-\ga)}{\Ga(n+1)}
$$
and therefore
$$
\int_0^{\ep} \frac{z^{\ga-1}}{(1+z)^{n+1}} d z
=\frac{\Ga(\ga)\Ga(n+1-\ga)}{\Ga(n+1)}(1+O((r+\ep)^{-n}))
$$
More generally, for a natural number $\b$ let us define
\begin{equation}
\lbl{eq.Ibc}
I_{\ga,\b}(n)=\int_0^\infty \frac{z^{\ga-1}}{(1+z)^{n+1}} (\log z)^{\b} d z
\end{equation}
Then, we have
\begin{equation}
\lbl{eq.im1}
I_{\ga,\b}(n)
=\frac{\Ga(\ga)\Ga(n+1-\ga)}{\Ga(n+1)} p_{\b}(\ga,n)
\end{equation}
where $p_{\b}(\ga,n)$ is a polynomial in the variables $\psi^{(k)}(n+1-\ga)$ 
and $\psi^{(l)}(\ga)$ with rational coefficients, where $\psi(z)=\Ga'(z)/\Ga(z)$
is the logarithmic derivative of the $\Ga$-function. For example,
we have:
\begin{eqnarray*}
p_0(n) &=& 1 \\
p_1(n) &=& -\psi(n+1-\ga)+\psi(\ga) \\
p_2(n) &=& \psi(n+1-\ga)^2 + \psi^{(1)}(n+1-\ga) -2 \psi(\ga) \psi(n+1-\ga)
+\psi(\ga)^2 + \psi^{(1)}(\ga)
\end{eqnarray*}
Compare also with \cite[Eqn.I.4.2]{M1} and \cite[Eqn.7.5]{M2}.
What is important is not the exact evaluation of $I_{\ga,\b}(n)$ given
in \eqref{eq.im1}, but the
fact that the sequence $I_{\ga,\b}(n)$ is of Nilsson type.
This follows from the fact that we have an asymptotic expansion (see \cite{O}):
\begin{equation}
\lbl{eq.gaexp}
\frac{\Ga(n+1-\ga)}{\Ga(n+1)} \sim \frac{1}{n^{\ga}}\left(1 
+ \frac{\ga^2-\ga}{2n} + \frac{ 3 \ga^4 - 2 \ga^3 - 3 \ga^2 +2 \ga}{24 n^2}
+ \dots \right)
\end{equation}
Alternatively, one may show that the sequence $I_{\ga,\b}(n)$ is of Nilsson type
by a change of variables $z=e^t-1$ which gives
$$
\int_0^\infty \frac{z^{\ga-1}}{(1+z)^{n+1}} (\log z)^{\b} d z=
\int_0^\infty e^{-nt} t^{\ga-1} A_{\ga,\b}(t) dt
$$
where
$$
A_{\ga,\b}(t) = \left(\frac{e^t-1}{t}\right)^{\ga-1}
\left(\log \left(\frac{e^t-1}{t}\right)-\log t\right)^{\b} dt
$$
is a function which can be expanded into a polynomial of $\log t$ with
coefficients functions which are analytic at $t=0$. Expand $A_{\ga,\b}(t)$ into
power series at $t=0$ and interchange summation and integration 
by applying Watson's lemma (see \cite{O}) to conclude that 
$I_{\ga,\b}(n)$ is of Nilsson type.

Replace $\D_{\l} G(\l z)$ by \eqref{eq.im3} in \eqref{eq.im2}
and interchange summation and integration by applying Watson's lemma 
(see \cite{O}). It follows that
$$
\frac{1}{2 \pi i} 
\int_{\calH_{\l}} \frac{G(z)}{z^{n+1}}dz 
\sim \l^{-n} \sum_{\a,\b} \frac{1}{n^{\a}} (\log n)^{\b} 
\sum_{k=0}^\infty c_{\a,\b,\l,k} \frac{1}{n^k} 
$$
Equation \eqref{eq.klm} cocnludes that $(a_n)$ is of Nilsson type.
Strictly speaking, the above analysis works only 
when $\Re(\a)>-1$. This is a local integrability assumption of the 
Beta-integral. The asymptotic expansion \eqref{eq.asexp} remains valid as 
stated even when $\Re(\a) \leq -1$ as follows by first integrating $G(z)$ 
a sufficient number of times, and then applying the analysis. 
This is exactly what was done in \cite[Sec.7]{CG}
at the cost of complicating the notation.
\end{proof}

\begin{remark}
\lbl{rem.stokes}
Since the sequence $(c'_{\a,\b,\l,k})$ as a function of $k$
is exponentially bounded and the asymptotic
expansion \eqref{eq.gaexp} is Gevrey-1, it follows that the sequence
$(c_{\a,\b,\l,k})$ is Gevrey-1. Moreover, if the sequence $(c'_{\a,\b,\l,k})$
lies in a number field $K$, then we can write the asymptotic expansion
of $(a_n)$ in the form \eqref{eq.nilsson} where $S_{\a,\b,\l}$ are polynomials
(with rational coefficients) of values of logarithmic derivatives of the
Gamma function at rational numbers.
\end{remark}

\section{$G$-functions}

\subsection{$G$-functions: examples of sequences of Nilsson type}
\lbl{sub.examples}

In \cite[Thm.1]{Ga2} it was proven that balanced multisum sequences
(which appear in abundance in Enumerative Combinatorics) are sequences
of Nilsson type. The proof uses the theory of $G$-functions which verifies
that the generating series of  balanced multisum sequences satisfies 
the hypothesis of Corollary \ref{cor.exists}. Let us give the definition
of a balanced multisum sequence, a $G$-function and an example.

\begin{definition}
\lbl{def.hyperg}
\rm{(a)}
A {\em term} $\ft_{n,k}$ in variables $(n,k)$ where $k=(k_1,\dots,k_r)$
is an expression of the form:
\begin{equation}
\lbl{eq.defterm}
\ft_{n,k}=C_0^n \prod_{i=1}^r C_i^{k_i} \prod_{j=1}^J A_j(n,k)!^{\e_j}
\end{equation}
where $C_i \in \overline{\BQ}$ for $i=0,\dots,r$, 
$\e_j=\pm 1$ for $j=1,\dots,J$, and $A_j$ are integral 
linear forms in the variables $(n,k)$ such that for every 
$n \in \BN$, the set 
\begin{equation}
\lbl{eq.kset}
\supp(\ft_{n,\bullet}):=
\{ k \in \BZ^r \, | \, A_j(n,k) \geq 0, \,\, j=1,\dots, J \}
\end{equation}
is finite. We will call a term {\em balanced} if in addition it
satisfies the {\em balance condition}:
\begin{equation}
\lbl{eq.Ajsum}
\sum_{j=1}^J \e_j A_j=0.
\end{equation}
\rm{(b)}
A {\em (balanced) multisum sequence} $(a_n)$ is a sequence of complex numbers
of the form
\begin{equation}
\lbl{eq.bc}
a_n=\sum_{k \in \supp(\ft_{n,\bullet})} \ft_{n,k}
\end{equation}
where $\ft$ is a (balanced) term and the sum is over a finite set that
depends on $\ft$.
\end{definition}

For example, the following sequence is a balanced multisum
\begin{equation}
\lbl{eq.apery}
a_n=\sum_{k,l} \binom{n}{k+l}^2 \binom{n+k}{k}^3 \binom{n+l}{l} =
\sum_{k,l} \frac{(n+k)!^3(n+l)!}{k!^3 l! n!^2 (k+l)!^2(n-k-l)!^2}
\end{equation}
where the summation is over the set of pairs of integers $(k,l)$
that satisfy $0 \leq k, l$ and $k+l \leq n$.

Let us now recall what us a $G$-function. The latter were 
introduced by Siegel in \cite{Si} with motivation being
arithmetic problems in elliptic integrals, and transcendence
problems in number theory. For further information about $G$-functions
and their properties, see \cite{An1,An2}.

\begin{definition}
\lbl{def.Gfunction}
We say that series $G(z)=\sum_{n=0}^\infty a_n z^n$ is a {\em $G$-function}
if 
\begin{itemize}
\item[(a)]
the coefficients $a_n$ are algebraic numbers,
\item[(b)]
there exists a constant $C>0$ so that for every $n \in \BN$
the absolute value of every conjugate of $a_n$ is less than or equal to 
$C^n$, 
\item[(c)]
the common denominator of $a_0,\dots, a_n$ is less than or equal 
to $C^n$,
\item[(d)]
$G(z)$ is holonomic, i.e., it satisfies a linear differential equation
with coefficients polynomials in $z$.
\end{itemize}
\end{definition}
$G$-functions satisfy the hypothesis of Corollary \ref{cor.exists};
see \cite{An1,An2}. Indeed, they satisfy a linear differential equation
which analytically continues them in the complex plane. Moreover, the
arithmetic hypothesis ensures that the local monodromy is quasi-unipotent.
We can now state the main result of \cite{Ga2}.

\begin{theorem}
\cite{Ga2}\lbl{thm.Ga2}
\rm{(a)} If $(a_n)$ is a balanced multisum sequence, its generating function
$G(z)=\sum_{n=0}^\infty a_n z^n$ is a $G$-function.
\newline
\rm{(b)} In that case, it follows that $(a_n)$ is a sequence of Nilsson
type.
\end{theorem}

The reader may have noticed that we defined the notion of a sequence
of Nilsson type only when $\limsup|a_n|^{1/n}>0$. In case the
generating series is a $G$-function, the remaining case
is taken care by the following lemma.

\begin{lemma}
\lbl{lem.Gpoly}
If $G(z)=\sum_{n=0}^\infty a_n z^n$ is a $G$-function and 
$\limsup|a_n|^{1/n}=0$,
then $a_n=0$ for all but finitely many $n$.
\end{lemma}

\begin{proof}
The assumption implies that $G(z)$ is an entire $G$-function. Since
those are regular-singular at infinity, it follows that $G(z)$ is
a polynomial; see also \cite{An1,An2}. The result follows.
\end{proof}

\subsection{Classical spin networks: examples of $G$-functions}
\lbl{sub.spin}

In \cite{GV1,GV2} it was proven that the evaluation of a quantum 
spin network at a fixed root of unity is a balanced multisum sequence,
and consequently it is a sequence of Nilsson type.

\section{Effective computations}
\lbl{sec.compute}

\subsection{Exact computations}
\lbl{sub.compute}

Proposition \ref{prop.unique} is a uniqueness statement about the
asymptotics of a sequence of Nilsson type, and Theorem \ref{thm.exists}
is an existence statement which is not effective. There are two types
of effective computations, exact and numerical. The exact computations
use as an input a linear recursion relation of the sequence. 
The following proposition is elementary and is discussed in detail
for example in \cite{FS,WZ2}.

\begin{proposition}
\lbl{prop.exactc}
Given a linear recursion relation for a sequence $(a_n)$ of Nilsson
type, one can compute exactly $\l,\a,\b$
and the power series $g_{\a,b,\l}(x)$ that appear in Equation 
\eqref{eq.nilsson}.
\end{proposition}
In particular, a linear recursion relation computes exactly the
asymptotics of a sequence of Nilsson type, up to a finite number of
{\em unknown Stokes constants}.

To apply Proposition \ref{prop.exactc} one needs to find 
a linear recursion for a sequence $(a_n)$. This comes
from the fundamental theorem of Wilf-Zeilberger which states that
a balanced multisum sequence is holonomic, i.e., satisfies a linear
recursion with coefficients polynomials in $n$; see \cite{Z,WZ1,PWZ}.
The proof of the above theorem has been computer implemented
and works well for single sums and reasonably well for double sums;
see \cite{PWZ,PR1,PR2}. As an example, consider the following sequence
from \cite[Sec.10]{GV1}
\begin{eqnarray*}
a_n &=& 
\frac{n!^6}{(3n+1)!^2}
\sum_{k = 3n}^{4n} \frac{(-1)^k(k+1)!}{(k-3n)!^4(4n-k)!^3} 
\end{eqnarray*}
Using the language of \cite{GV1}, $(a_n)$ is the evaluation of the
tetrahedron spin network (also known as $6j$-symbol) when all edges
are equal to $n$.
The command 
$$
\psdraw{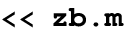}{0.5in}
$$
loads the package of \cite{PR2} into {\tt Mathematica}. The command
$$
\psdraw{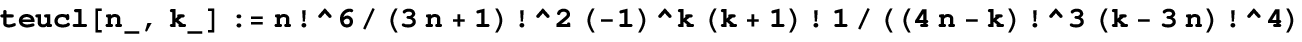}{5in}
$$
defines the summand of the sequence $(a_n)$, and the command
$$
\psdraw{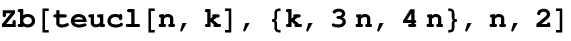}{2in}
$$
computes the following second order linear recursion relation for the
sequence $(a_n)$
$$
\psdraw{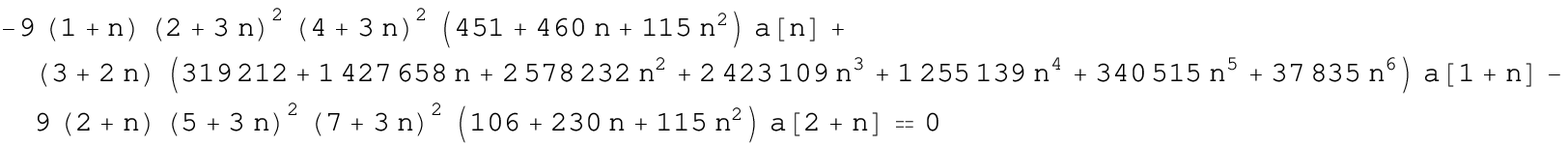}{6in}
$$
This linear recursion has two formal power series solutions of the
form

\begin{eqnarray*}
a_{\pm,n} &=&
\frac{1}{n^{3/2}}\La_{\pm}^n 
\left(1+
\frac{-432 \pm 31 i \sqrt{2}}{576 n}+
\frac{109847 \mp 22320 i \sqrt{2}}{331776 n^2}+
\frac{-18649008 \pm 4914305 i \sqrt{2}}{573308928 n^3} 
\right. \\
& + & \left.
\frac{14721750481 \pm 45578388960 i \sqrt{2}}{660451885056 n^4}+
\frac{-83614134803760 \pm 7532932167923 i \sqrt{2}}{380420285792256 n^5}
\right. \\
& + & \left.
\frac{-31784729861796581 \mp 212040612888146640 i \sqrt{2}}{
657366253849018368 n^6}+
O\left(\frac{1}{n^7}\right)
\right)
\end{eqnarray*}
where
$$
\La_{\pm}=\frac{329 \mp 460 i \sqrt{2}}{729}=e^{\mp i 6 \arccos(1/3)}
$$
are two complex numbers of absolute value $1$. 
The coefficients of the formal power series $a_{\pm,n}$ are in the number
field $K=Q(\sqrt{-2})$.

\subsection{Numerical computations}
\lbl{sub.num}

When a sequence $(a_n)$ is given by a multi-dimensional balanced sum,
the computed implemented WZ method may not terminate. In that case,
one may develop numerical methods for finding $\l, \a,\b$ as in Equation
\eqref{eq.nilsson}. An example of this method is the asymptotics of the
evaluation of the Cube Spin Network that appears in the Appendix of \cite{GV1}.
Effective methods for numerical computations of asymptotics have been 
developed by several authors, and have also been studied by Zagier.

\subsection{Acknowledgment}
The idea of the present paper was conceived during the New York
Conference on {\em Interactions between Hyperbolic Geometry, 
Quantum Topology and Number Theory} in New York in the summer of 2009.
The final writing of the paper occurred in Oberwolfach in the summer of 2010.
The author wishes to thank the organizers of the New York Conference, 
A. Champanerkar, O. Dasbach, E. Kalfagianni, I. Kofman, W. Neumann and 
N. Stoltzfus, and the organizers of the Oberwolfach Conference,
P. Gunnells, W. Neumann, A. Sikora, D. Zagier, for their hospitality. 
In addition, the author wishes to thank E. Croot, D. Zagier, D. Zeilberger
for stimulating conversations and R. van der Veen for a careful
reading of the manuscript.

\ifx\undefined\bysame
        \newcommand{\bysame}{\leavevmode\hbox
to3em{\hrulefill}\,}
\fi

\end{document}